\documentclass{amsart}
\usepackage{amsmath}
\usepackage{amsfonts}

\newtheorem{theorem}{Theorem}
\theoremstyle{plain}

\newtheorem{corollary}{Corollary}

\newtheorem{definition}{Definition}

\newtheorem{lemma}{Lemma}

\newtheorem{proposition}{Proposition}

\numberwithin{equation}{section}
\input{tcilatex}

\begin{document}
\title[Harmonically Convex Functions]{Hermite-Hadamard and Simpson-like type
inequalities for differentiable harmonically convex functions}
\author{\.{I}mdat \.{I}\c{s}can}
\address{Department of Mathematics, Faculty of Arts and Sciences,\\
Giresun University, 28100, Giresun, Turkey.}
\email{imdati@yahoo.com, imdat.iscan@giresun.edu.tr}

\begin{abstract}
In this paper, a new identity for differentiable functions is derived. A
consequence of the identity is that the author establishes some new general
inequalities containing all of the Hermite-Hadamard and Simpson-like type
for functions whose derivatives in absolute value at certain power are
harmonically convex. Some applications to special means of real numbers are
also given.
\end{abstract}

\maketitle

\section{Introduction}

Let $f:I\subset \mathbb{R\rightarrow R}$ be a convex function defined on the
interval $I$ of real numbers and $a,b\in I$ with $a<b$. The following
inequality%
\begin{equation}
f\left( \frac{a+b}{2}\right) \leq \frac{1}{b-a}\dint\limits_{a}^{b}f(x)dx%
\leq \frac{f(a)+f(b)}{2}  \label{1-1}
\end{equation}

holds. This double inequality is known in the literature as Hermite-Hadamard
integral inequality for convex functions. Note that some of the classical
inequalities for means can be derived from (\ref{1-1}) for appropriate
particular selections of the mapping $f$. Both inequalities hold in the
reversed direction if $f$ is concave.

Following inequality is well known in the literature as Simpson inequality:

\begin{theorem}
Let $f:\left[ a,b\right] \mathbb{\rightarrow R}$ be a four times
continuously differentiable mapping on $\left( a,b\right) $ and $\left\Vert
f^{(4)}\right\Vert _{\infty }=\underset{x\in \left( a,b\right) }{\sup }%
\left\vert f^{(4)}(x)\right\vert <\infty .$ Then the following inequality
holds:%
\begin{equation*}
\left\vert \frac{1}{3}\left[ \frac{f(a)+f(b)}{2}+2f\left( \frac{a+b}{2}%
\right) \right] -\frac{1}{b-a}\dint\limits_{a}^{b}f(x)dx\right\vert \leq 
\frac{1}{2880}\left\Vert f^{(4)}\right\Vert _{\infty }\left( b-a\right) ^{4}.
\end{equation*}
\end{theorem}

For some results which generalize, improve and extend the Hermite-Hadamard
and Simpson inequalities, we refer the reader to the recent papers (see \cite%
{DP00,I13,I13a,I13b,K04,XQ12,YHT04} ).

In \cite{I13c}, the author introduced the concept of harmonically convex
functions and established some results connected with the right-hand side of
new inequalities similar to the inequality (\ref{1-1}) for these classes of
functions. Some applications to special means of positive real numbers are
also given.

\begin{definition}
Let $I\subset 
\mathbb{R}
\backslash \left\{ 0\right\} $ be a real interval. A function $%
f:I\rightarrow 
\mathbb{R}
$ is said to be harmonically convex, if \ 
\begin{equation}
f\left( \frac{xy}{tx+(1-t)y}\right) \leq tf(y)+(1-t)f(x)  \label{1-2}
\end{equation}%
for all $x,y\in I$ and $t\in \lbrack 0,1]$. If the inequality in (\ref{1-2})
is reversed, then $f$ is said to be harmonically concave.
\end{definition}

The following result of the Hermite-Hadamard type holds.

\begin{theorem}
Let $f:I\subset 
\mathbb{R}
\backslash \left\{ 0\right\} \rightarrow 
\mathbb{R}
$ be a harmonically convex function and $a,b\in I$ with $a<b.$ If $f\in
L[a,b]$ then the following inequalities hold 
\begin{equation}
f\left( \frac{2ab}{a+b}\right) \leq \frac{ab}{b-a}\dint\limits_{a}^{b}\frac{%
f(x)}{x^{2}}dx\leq \frac{f(a)+f(b)}{2}.  \label{1-3}
\end{equation}%
The \ above inequalities are sharp.
\end{theorem}

Some results connected with the right part of (\ref{1-3}) was given in \cite%
{I13c} as follows:

\begin{theorem}
Let $f:I\subset \left( 0,\infty \right) \rightarrow 
\mathbb{R}
$ be a differentiable function on $I^{\circ }$, $a,b\in I$ with $a<b,$ and $%
f^{\prime }\in L[a,b].$ If $\left\vert f^{\prime }\right\vert ^{q}$ is
harmonically convex on $[a,b]$ for $q\geq 1,$ then%
\begin{eqnarray}
&&\left\vert \frac{f(a)+f(b)}{2}-\frac{ab}{b-a}\dint\limits_{a}^{b}\frac{f(x)%
}{x^{2}}dx\right\vert  \label{1-4} \\
&\leq &\frac{ab\left( b-a\right) }{2}\lambda _{1}^{1-\frac{1}{q}}\left[
\lambda _{2}\left\vert f^{\prime }\left( a\right) \right\vert ^{q}+\lambda
_{3}\left\vert f^{\prime }\left( b\right) \right\vert ^{q}\right] ^{\frac{1}{%
q}},  \notag
\end{eqnarray}%
where 
\begin{eqnarray*}
\lambda _{1} &=&\frac{1}{ab}-\frac{2}{\left( b-a\right) ^{2}}\ln \left( 
\frac{\left( a+b\right) ^{2}}{4ab}\right) , \\
\lambda _{2} &=&\frac{-1}{b\left( b-a\right) }+\frac{3a+b}{\left( b-a\right)
^{3}}\ln \left( \frac{\left( a+b\right) ^{2}}{4ab}\right) , \\
\lambda _{3} &=&\frac{1}{a\left( b-a\right) }-\frac{3b+a}{\left( b-a\right)
^{3}}\ln \left( \frac{\left( a+b\right) ^{2}}{4ab}\right) \\
&=&\lambda _{1}-\lambda _{2}.
\end{eqnarray*}
\end{theorem}

\begin{theorem}
Let $f:I\subset \left( 0,\infty \right) \rightarrow 
\mathbb{R}
$ be a differentiable function on $I^{\circ }$, $a,b\in I$ with $a<b,$ and $%
f^{\prime }\in L[a,b].$ If $\left\vert f^{\prime }\right\vert ^{q}$ is
harmonically convex on $[a,b]$ for $q>1,\;\frac{1}{p}+\frac{1}{q}=1,$ then%
\begin{eqnarray}
&&\left\vert \frac{f(a)+f(b)}{2}-\frac{ab}{b-a}\dint\limits_{a}^{b}\frac{f(x)%
}{x^{2}}dx\right\vert  \label{1-5} \\
&\leq &\frac{ab\left( b-a\right) }{2}\left( \frac{1}{p+1}\right) ^{\frac{1}{p%
}}\left( \mu _{1}\left\vert f^{\prime }\left( a\right) \right\vert ^{q}+\mu
_{2}\left\vert f^{\prime }\left( b\right) \right\vert ^{q}\right) ^{\frac{1}{%
q}},  \notag
\end{eqnarray}%
where%
\begin{eqnarray*}
\mu _{1} &=&\frac{\left[ a^{2-2q}+b^{1-2q}\left[ \left( b-a\right) \left(
1-2q\right) -a\right] \right] }{2\left( b-a\right) ^{2}\left( 1-q\right)
\left( 1-2q\right) }, \\
\mu _{2} &=&\frac{\left[ b^{2-2q}-a^{1-2q}\left[ \left( b-a\right) \left(
1-2q\right) +b\right] \right] }{2\left( b-a\right) ^{2}\left( 1-q\right)
\left( 1-2q\right) }.
\end{eqnarray*}
\end{theorem}

In this paper, we shall give some general integral inequalities connected
with the left and right parts of (\ref{1-3}), as a result of this, we shall
obtained some new midpoint, trapezoid and Simpson like-type inequalities for
differentiable harmonically convex functions.

\section{Main results}

In order to prove our main resuls we need the following lemma:

\begin{lemma}
\label{2.1}Let $f:I\subset 
\mathbb{R}
\backslash \left\{ 0\right\} \rightarrow 
\mathbb{R}
$ be a differentiable function on $I^{\circ }$ and $a,b\in I$ with $a<b$. If 
$f^{\prime }\in L[a,b]$ then for $\lambda \in \left[ 0,1\right] $ we have
the equality%
\begin{eqnarray*}
&&\left( 1-\lambda \right) f\left( \frac{2ab}{a+b}\right) +\lambda \left( 
\frac{f(a)+f(b)}{2}\right) -\frac{ab}{b-a}\dint\limits_{a}^{b}\frac{f(x)}{%
x^{2}}dx \\
&=&\frac{ab\left( b-a\right) }{2}\left[ \dint\limits_{0}^{1/2}\frac{\lambda
-2t}{A_{t}^{2}}f^{\prime }\left( \frac{ab}{A_{t}}\right)
dt+\dint\limits_{1/2}^{1}\frac{2-\lambda -2t}{A_{t}^{2}}f^{\prime }\left( 
\frac{ab}{A_{t}}\right) dt\right] ,
\end{eqnarray*}%
where $A_{t}=tb+(1-t)a.$
\end{lemma}

\begin{proof}
It suffices to note that%
\begin{eqnarray*}
I_{1} &=&ab\left( b-a\right) \dint\limits_{0}^{1}\frac{\lambda -2t}{A_{t}^{2}%
}f^{\prime }\left( \frac{ab}{A_{t}}\right) dt \\
&=&\left. \left( 2t-\lambda \right) f\left( \frac{ab}{A_{t}}\right)
\right\vert _{0}^{1/2}-2\dint\limits_{0}^{1/2}f\left( \frac{ab}{A_{t}}%
\right) dt \\
&=&\left( 1-\lambda \right) f\left( \frac{2ab}{a+b}\right) +\lambda
f(b)-2\dint\limits_{0}^{1/2}f\left( \frac{ab}{A_{t}}\right) dt.
\end{eqnarray*}%
Setting $x=\frac{ab}{A_{t}}$ and $dx=\frac{-ab\left( b-a\right) }{A_{t}^{2}}%
dt$, which gives%
\begin{equation*}
I_{1}=\left( 1-\lambda \right) f\left( \frac{2ab}{a+b}\right) +\lambda f(b)-%
\frac{2ab}{b-a}\dint\limits_{2ab/(a+b)}^{b}\frac{f\left( x\right) }{x^{2}}dx.
\end{equation*}%
Similarly, we can show that%
\begin{eqnarray*}
I_{2} &=&ab\left( b-a\right) \dint\limits_{1/2}^{1}\frac{2-\lambda -2t}{%
A_{t}^{2}}f^{\prime }\left( \frac{ab}{A_{t}}\right) dt \\
&=&\lambda f(a)+\left( 1-\lambda \right) f\left( \frac{2ab}{a+b}\right) -%
\frac{2ab}{b-a}\dint\limits_{a}^{2ab/(a+b)}\frac{f\left( x\right) }{x^{2}}dx.
\end{eqnarray*}%
Thus,%
\begin{equation*}
\frac{I_{1}+I_{2}}{2}=\left( 1-\lambda \right) f\left( \frac{2ab}{a+b}%
\right) +\lambda \left( \frac{f(a)+f(b)}{2}\right) -\frac{ab}{b-a}%
\dint\limits_{a}^{b}\frac{f(x)}{x^{2}}dx
\end{equation*}%
which is required.
\end{proof}

\begin{theorem}
\label{2.2}Let $f:I\subset \left( 0,\infty \right) \rightarrow 
\mathbb{R}
$ be a differentiable function on $I^{\circ }$, $a,b\in I$ with $a<b,$ and $%
f^{\prime }\in L[a,b].$ If $\left\vert f^{\prime }\right\vert ^{q}$ is
harmonically convex on $[a,b]$ for $q\geq 1$ and then we have the following
inequality for $\lambda \in \left[ 0,1\right] $%
\begin{equation}
\left\vert \left( 1-\lambda \right) f\left( \frac{2ab}{a+b}\right) +\lambda
\left( \frac{f(a)+f(b)}{2}\right) -\frac{ab}{b-a}\dint\limits_{a}^{b}\frac{%
f(x)}{x^{2}}dx\right\vert  \label{2-2}
\end{equation}%
\begin{eqnarray*}
&&\leq \frac{ab\left( b-a\right) }{2}\left\{ C_{1}^{1-\frac{1}{q}}(\lambda
;a,b)\left[ C_{2}(\lambda ;a,b)\left\vert f^{\prime }\left( a\right)
\right\vert ^{q}+C_{3}(\lambda ;a,b)\left\vert f^{\prime }\left( b\right)
\right\vert ^{q}\right] ^{\frac{1}{q}}\right. \\
&&\left. +C_{1}^{1-\frac{1}{q}}(\lambda ;b,a)\left[ C_{3}(\lambda
;b,a)\left\vert f^{\prime }\left( a\right) \right\vert ^{q}+C_{2}(\lambda
;b,a)\left\vert f^{\prime }\left( b\right) \right\vert ^{q}\right] ^{\frac{1%
}{q}}\right\} ,
\end{eqnarray*}%
where%
\begin{eqnarray*}
&&C_{1}(\lambda ;u,\vartheta )=\frac{1}{\left( \vartheta -u\right) ^{2}} \\
&&\times \left[ -4+\frac{\left[ \lambda \left( \vartheta -u\right) +2u\right]
\left( 3u+\vartheta \right) }{u\left( u+\vartheta \right) }+2\ln \left( 
\frac{2u\left( u+\vartheta \right) }{\left( 2u+\lambda \left( \vartheta
-u\right) \right) ^{2}}\right) \right] ,
\end{eqnarray*}%
\begin{eqnarray*}
&&C_{2}(\lambda ;u,\vartheta )=\frac{1}{\left( \vartheta -u\right) ^{3}} \\
&&\times \left\{ \left[ \lambda \left( \vartheta -u\right) +4u\right] \ln
\left( \frac{\left[ \lambda \left( \vartheta -u\right) +2u\right] ^{2}}{%
2u\left( u+\vartheta \right) }\right) \right. \\
&&\left. -\frac{\left[ \lambda \left( \vartheta -u\right) +2u\right] \left(
5u+3\vartheta \right) }{u+\vartheta }+7u+\vartheta \right\} ,
\end{eqnarray*}%
and%
\begin{equation*}
C_{3}(\lambda ;u,\vartheta )=C_{1}(\lambda ;u,\vartheta )-C_{2}(\lambda
;u,\vartheta ),\ u,\vartheta >0.
\end{equation*}
\end{theorem}

\begin{proof}
Let $A_{t}=tb+(1-t)a$. From Lemma \ref{2.1} and using the power mean
inequality, we have%
\begin{eqnarray*}
&&\left\vert \left( 1-\lambda \right) f\left( \frac{2ab}{a+b}\right)
+\lambda \left( \frac{f(a)+f(b)}{2}\right) -\frac{ab}{b-a}%
\dint\limits_{a}^{b}\frac{f(x)}{x^{2}}dx\right\vert \\
&\leq &\frac{ab\left( b-a\right) }{2}\left[ \dint\limits_{0}^{1/2}\frac{%
\left\vert \lambda -2t\right\vert }{A_{t}^{2}}\left\vert f^{\prime }\left( 
\frac{ab}{A_{t}}\right) \right\vert dt+\dint\limits_{1/2}^{1}\frac{%
\left\vert 2-\lambda -2t\right\vert }{A_{t}^{2}}\left\vert f^{\prime }\left( 
\frac{ab}{A_{t}}\right) \right\vert dt\right]
\end{eqnarray*}%
\begin{eqnarray*}
&\leq &\frac{ab\left( b-a\right) }{2}\left\{ \left( \dint\limits_{0}^{1/2}%
\frac{\left\vert \lambda -2t\right\vert }{A_{t}^{2}}dt\right) ^{1-\frac{1}{q}%
}\left( \dint\limits_{0}^{1/2}\frac{\left\vert \lambda -2t\right\vert }{%
A_{t}^{2}}\left\vert f^{\prime }\left( \frac{ab}{A_{t}}\right) \right\vert
^{q}dt\right) ^{\frac{1}{q}}\right. \\
&&\left. +\left( \dint\limits_{1/2}^{1}\frac{\left\vert 2-\lambda
-2t\right\vert }{A_{t}^{2}}dt\right) ^{1-\frac{1}{q}}\left(
\dint\limits_{1/2}^{1}\frac{\left\vert 2-\lambda -2t\right\vert }{A_{t}^{2}}%
\left\vert f^{\prime }\left( \frac{ab}{A_{t}}\right) \right\vert
^{q}dt\right) ^{\frac{1}{q}}\right\} .
\end{eqnarray*}%
Hence, by harmonically convexity of $\left\vert f^{\prime }\right\vert ^{q}$
on $[a,b],$ we have%
\begin{equation*}
\left\vert \left( 1-\lambda \right) f\left( \frac{2ab}{a+b}\right) +\lambda
\left( \frac{f(a)+f(b)}{2}\right) -\frac{ab}{b-a}\dint\limits_{a}^{b}\frac{%
f(x)}{x^{2}}dx\right\vert \leq \frac{ab\left( b-a\right) }{2}
\end{equation*}%
\begin{eqnarray*}
&&\times \left\{ \left( \dint\limits_{0}^{1/2}\frac{\left\vert \lambda
-2t\right\vert }{A_{t}^{2}}dt\right) ^{1-\frac{1}{q}}\left(
\dint\limits_{0}^{1/2}\frac{\left\vert \lambda -2t\right\vert \left[
t\left\vert f^{\prime }\left( a\right) \right\vert ^{q}+(1-t)\left\vert
f^{\prime }\left( b\right) \right\vert ^{q}\right] }{A_{t}^{2}}dt\right) ^{%
\frac{1}{q}}\right. \\
&&\left. +\left( \dint\limits_{1/2}^{1}\frac{\left\vert 2-\lambda
-2t\right\vert }{A_{t}^{2}}dt\right) ^{1-\frac{1}{q}}\left(
\dint\limits_{1/2}^{1}\frac{\left\vert 2-\lambda -2t\right\vert \left[
t\left\vert f^{\prime }\left( a\right) \right\vert ^{q}+(1-t)\left\vert
f^{\prime }\left( b\right) \right\vert ^{q}\right] }{A_{t}^{2}}dt\right) ^{%
\frac{1}{q}}\right\}
\end{eqnarray*}%
\begin{eqnarray*}
&\leq &\frac{ab\left( b-a\right) }{2}C_{1}^{1-\frac{1}{q}}(\lambda
;a,b)\left\{ \left[ C_{2}(\lambda ;a,b)\left\vert f^{\prime }\left( a\right)
\right\vert ^{q}+C_{3}(\lambda ;a,b)\left\vert f^{\prime }\left( b\right)
\right\vert ^{q}\right] ^{\frac{1}{q}}\right. \\
&&\left. \left[ C_{3}(\lambda ;b,a)\left\vert f^{\prime }\left( a\right)
\right\vert ^{q}+C_{2}(\lambda ;b,a)\left\vert f^{\prime }\left( b\right)
\right\vert ^{q}\right] ^{\frac{1}{q}}\right\} .
\end{eqnarray*}%
It is easily check that 
\begin{eqnarray*}
&&\dint\limits_{0}^{1/2}\frac{\left\vert \lambda -2t\right\vert }{A_{t}^{2}}%
dt=C_{1}(\lambda ;a,b)=\frac{1}{\left( b-a\right) ^{2}} \\
&&\times \left[ -4+\frac{\left[ \lambda \left( b-a\right) +2a\right] \left(
3a+b\right) }{a\left( a+b\right) }+2\ln \left( \frac{2a\left( a+b\right) }{%
\left( 2a+\lambda \left( b-a\right) \right) ^{2}}\right) \right] ,
\end{eqnarray*}%
\begin{eqnarray*}
&&\dint\limits_{0}^{1/2}\frac{\left\vert \lambda -2t\right\vert t}{A_{t}^{2}}%
dt=C_{2}(\lambda ;a,b)=\frac{1}{\left( b-a\right) ^{3}} \\
&&\times \left\{ \left[ \lambda \left( b-a\right) +4a\right] \ln \left( 
\frac{\left[ \lambda \left( b-a\right) +2a\right] ^{2}}{2a\left( a+b\right) }%
\right) \right. \\
&&\left. -\frac{\left[ \lambda \left( b-a\right) +2a\right] \left(
5a+3b\right) }{a+b}+7a+b\right\} ,
\end{eqnarray*}%
\begin{equation*}
\dint\limits_{0}^{1/2}\frac{\left\vert \lambda -2t\right\vert (1-t)}{%
A_{t}^{2}}dt=C_{3}(\lambda ;a,b)=C_{1}(\lambda ;a,b)-C_{2}(\lambda ;a,b),
\end{equation*}%
\begin{equation*}
\dint\limits_{1/2}^{1}\frac{\left\vert 2-\lambda -2t\right\vert }{A_{t}^{2}}%
dt=C_{1}(\lambda ;b,a),\ \dint\limits_{1/2}^{1}\frac{\left\vert 2-\lambda
-2t\right\vert (1-t)}{A_{t}^{2}}dt=C_{2}(\lambda ;b,a),
\end{equation*}%
and%
\begin{equation*}
\dint\limits_{1/2}^{1}\frac{\left\vert 2-\lambda -2t\right\vert t}{A_{t}^{2}}%
dt=C_{3}(\lambda ;b,a)=C_{1}(\lambda ;b,a)-C_{2}(\lambda ;b,a).
\end{equation*}%
This concludes the proof.
\end{proof}

\begin{corollary}
Under the assumptions Theorem \ref{2.2} with $\lambda =0$, we have%
\begin{equation*}
\left\vert f\left( \frac{2ab}{a+b}\right) -\frac{ab}{b-a}\dint\limits_{a}^{b}%
\frac{f(x)}{x^{2}}dx\right\vert \leq \frac{ab\left( b-a\right) }{2}
\end{equation*}%
\begin{eqnarray*}
&&\times \left\{ C_{1}^{1-\frac{1}{q}}(0;a,b)\left[ C_{2}(0;a,b)\left\vert
f^{\prime }\left( a\right) \right\vert ^{q}+C_{3}(0;a,b)\left\vert f^{\prime
}\left( b\right) \right\vert ^{q}\right] ^{\frac{1}{q}}\right. \\
&&\left. C_{1}^{1-\frac{1}{q}}(0;b,a)\left[ C_{3}(0;b,a)\left\vert f^{\prime
}\left( a\right) \right\vert ^{q}+C_{2}(0;b,a)\left\vert f^{\prime }\left(
b\right) \right\vert ^{q}\right] ^{\frac{1}{q}}\right\} ,
\end{eqnarray*}%
where%
\begin{eqnarray*}
C_{1}(0;u,\vartheta ) &=&\frac{2}{\left( \vartheta -u\right) ^{2}}\left[ \ln
\left( \frac{u+\vartheta }{2u}\right) -\frac{\vartheta -u}{u+\vartheta }%
\right] , \\
C_{2}(0;u,\vartheta ) &=&\frac{1}{\left( \vartheta -u\right) ^{3}}\left[ 
\frac{\left( 3u+\vartheta \right) \left( \vartheta -u\right) }{u+\vartheta }%
+4u\ln \left( \frac{2u}{u+\vartheta }\right) \right] , \\
C_{3}(0;u,\vartheta ) &=&\frac{1}{\left( \vartheta -u\right) ^{2}}\left[ 
\frac{2\left( u+\vartheta \right) }{\vartheta -u}\ln \left( \frac{%
u+\vartheta }{2u}\right) -\frac{u+3\vartheta }{u+\vartheta }\right] ,\
u,\vartheta >0.
\end{eqnarray*}
\end{corollary}

\begin{corollary}
Under the assumptions Theorem \ref{2.2} with $\lambda =1$, we have%
\begin{equation*}
\left\vert \frac{f(a)+f(b)}{2}-\frac{ab}{b-a}\dint\limits_{a}^{b}\frac{f(x)}{%
x^{2}}dx\right\vert \leq \frac{ab\left( b-a\right) }{2}
\end{equation*}%
\begin{eqnarray*}
&&\times \left\{ C_{1}^{1-\frac{1}{q}}(1;a,b)\left[ C_{2}(1;a,b)\left\vert
f^{\prime }\left( a\right) \right\vert ^{q}+C_{3}(1;a,b)\left\vert f^{\prime
}\left( b\right) \right\vert ^{q}\right] ^{\frac{1}{q}}\right. \\
&&\left. C_{1}^{1-\frac{1}{q}}(1;b,a)\left[ C_{3}(1;b,a)\left\vert f^{\prime
}\left( a\right) \right\vert ^{q}+C_{2}(1;b,a)\left\vert f^{\prime }\left(
b\right) \right\vert ^{q}\right] ^{\frac{1}{q}}\right\} ,
\end{eqnarray*}%
where%
\begin{eqnarray*}
C_{1}(1;u,\vartheta ) &=&\frac{1}{\left( \vartheta -u\right) ^{2}}\left[ 
\frac{\vartheta -u}{u}+2\ln \left( \frac{2u}{u+\vartheta }\right) \right] ,
\\
C_{2}(1;u,\vartheta ) &=&\frac{1}{\left( \vartheta -u\right) ^{3}}\left[
\left( 3u+\vartheta \right) \ln \left( \frac{u+\vartheta }{2u}\right)
-2(\vartheta -u)\right] , \\
C_{3}(1;u,\vartheta ) &=&\frac{1}{\left( \vartheta -u\right) ^{2}}\left[ 
\frac{u+\vartheta }{u}-\frac{u+3\vartheta }{\vartheta -u}\ln \left( \frac{%
u+\vartheta }{2u}\right) \right] ,\ u,\vartheta >0.
\end{eqnarray*}
\end{corollary}

\begin{corollary}
Under the assumptions Theorem \ref{2.2} with $\lambda =1/3$, we have%
\begin{equation*}
\left\vert \frac{1}{3}\left[ \frac{f(a)+f(b)}{2}+2f\left( \frac{2ab}{a+b}%
\right) \right] -\frac{ab}{b-a}\dint\limits_{a}^{b}\frac{f(x)}{x^{2}}%
dx\right\vert \leq \frac{ab\left( b-a\right) }{2}
\end{equation*}%
\begin{eqnarray*}
&&\times \left\{ C_{1}^{1-\frac{1}{q}}(1/3;a,b)\left[ C_{2}(1/3;a,b)\left%
\vert f^{\prime }\left( a\right) \right\vert ^{q}+C_{3}(1/3;a,b)\left\vert
f^{\prime }\left( b\right) \right\vert ^{q}\right] ^{\frac{1}{q}}\right. \\
&&\left. C_{1}^{1-\frac{1}{q}}(1/3;b,a)\left[ C_{3}(1/3;b,a)\left\vert
f^{\prime }\left( a\right) \right\vert ^{q}+C_{2}(1/3;b,a)\left\vert
f^{\prime }\left( b\right) \right\vert ^{q}\right] ^{\frac{1}{q}}\right\} ,
\end{eqnarray*}%
where%
\begin{eqnarray*}
C_{1}(1/3;u,\vartheta ) &=&\frac{1}{\left( \vartheta -u\right) ^{2}}\left[ 
\frac{\left( \vartheta -u\right) \left( \vartheta -3u\right) }{3u\left(
u+\vartheta \right) }+2\ln \left( \frac{18u\left( u+\vartheta \right) }{%
\left( 5u+\vartheta \right) ^{2}}\right) \right] , \\
C_{2}(1/3;u,\vartheta ) &=&\frac{1}{\left( \vartheta -u\right) ^{3}}\left[
\left( \frac{11u+\vartheta }{3}\right) \ln \left( \frac{\left( 5u+\vartheta
\right) ^{2}}{18u\left( u+\vartheta \right) }\right) +\frac{4u(\vartheta -u)%
}{3\left( u+\vartheta \right) }\right] , \\
C_{3}(1/3;u,\vartheta ) &=&\frac{1}{\left( \vartheta -u\right) ^{2}}\left[ 
\frac{\vartheta ^{2}-4u\vartheta -u^{2}}{3u\left( u+\vartheta \right) }+%
\frac{5u+7\vartheta }{3\left( \vartheta -u\right) }\ln \left( \frac{%
18u\left( u+\vartheta \right) }{\left( 5u+\vartheta \right) ^{2}}\right) %
\right] ,\ u,\vartheta >0.
\end{eqnarray*}
\end{corollary}

\begin{theorem}
\label{2.3}Let $f:I\subset \left( 0,\infty \right) \rightarrow 
\mathbb{R}
$ be a differentiable function on $I^{\circ }$, $a,b\in I$ with $a<b,$ and $%
f^{\prime }\in L[a,b].$ If $\left\vert f^{\prime }\right\vert ^{q}$ is
harmonically convex on $[a,b]$ for $q>1$ and then we have the following
inequality for $\lambda \in \left[ 0,1\right] $%
\begin{equation}
\left\vert \left( 1-\lambda \right) f\left( \frac{2ab}{a+b}\right) +\lambda
\left( \frac{f(a)+f(b)}{2}\right) -\frac{ab}{b-a}\dint\limits_{a}^{b}\frac{%
f(x)}{x^{2}}dx\right\vert \leq \frac{ab\left( b-a\right) }{2}  \label{2-3}
\end{equation}%
\begin{equation*}
\times \left\{ C_{1}^{\frac{1}{p}}(\lambda ,p;a,b)\left( \frac{\left\vert
f^{\prime }\left( \frac{2ab}{a+b}\right) \right\vert ^{q}+\left\vert
f^{\prime }\left( b\right) \right\vert ^{q}}{4}\right) ^{\frac{1}{q}}+C_{1}^{%
\frac{1}{p}}(\lambda ,p;b,a)\left( \frac{\left\vert f^{\prime }\left( \frac{%
2ab}{a+b}\right) \right\vert ^{q}+\left\vert f^{\prime }\left( a\right)
\right\vert ^{q}}{4}\right) ^{\frac{1}{q}}\right\} ,
\end{equation*}%
where%
\begin{equation*}
C_{1}(\lambda ,p;u,\vartheta )=\dint\limits_{0}^{1/2}\frac{\left\vert
\lambda -2t\right\vert ^{p}}{\left( tb+(1-t)a\right) ^{2p}}dt,\ u,\vartheta
>0.
\end{equation*}%
and $1/p+1/q=1$.
\end{theorem}

\begin{proof}
Let $A_{t}=tb+(1-t)a$. Using Lemma \ref{2.1} and H\"{o}lder's integral
inequality, we deduce%
\begin{eqnarray*}
&&\left\vert \left( 1-\lambda \right) f\left( \frac{2ab}{a+b}\right)
+\lambda \left( \frac{f(a)+f(b)}{2}\right) -\frac{ab}{b-a}%
\dint\limits_{a}^{b}\frac{f(x)}{x^{2}}dx\right\vert \\
&\leq &\frac{ab\left( b-a\right) }{2}\left[ \dint\limits_{0}^{1/2}\frac{%
\left\vert \lambda -2t\right\vert }{A_{t}^{2}}\left\vert f^{\prime }\left( 
\frac{ab}{A_{t}}\right) \right\vert dt+\dint\limits_{1/2}^{1}\frac{%
\left\vert 2-\lambda -2t\right\vert }{A_{t}^{2}}\left\vert f^{\prime }\left( 
\frac{ab}{A_{t}}\right) \right\vert dt\right]
\end{eqnarray*}%
\begin{eqnarray}
&\leq &\frac{ab\left( b-a\right) }{2}\left\{ \left( \dint\limits_{0}^{1/2}%
\frac{\left\vert \lambda -2t\right\vert ^{p}}{A_{t}^{2p}}dt\right) ^{\frac{1%
}{p}}\left( \dint\limits_{0}^{1/2}\left\vert f^{\prime }\left( \frac{ab}{%
A_{t}}\right) \right\vert ^{q}dt\right) ^{\frac{1}{q}}\right.  \label{2-3a}
\\
&&\left. +\left( \dint\limits_{1/2}^{1}\frac{\left\vert 2-\lambda
-2t\right\vert ^{p}}{A_{t}^{2p}}dt\right) ^{\frac{1}{p}}\left(
\dint\limits_{1/2}^{1}\left\vert f^{\prime }\left( \frac{ab}{A_{t}}\right)
\right\vert ^{q}dt\right) ^{\frac{1}{q}}\right\} .  \notag
\end{eqnarray}%
Using the harmonically convexity of $\left\vert f^{\prime }\right\vert ^{q}$%
, we obtain the following inequalities from inequality (\ref{1-3}): 
\begin{eqnarray}
\dint\limits_{0}^{1/2}\left\vert f^{\prime }\left( \frac{ab}{A_{t}}\right)
\right\vert ^{q}dt &\leq &\frac{1}{2}\left( \frac{2ab}{b-a}\dint\limits_{%
\frac{2ab}{a+b}}^{b}\frac{\left\vert f^{\prime }\left( x\right) \right\vert
^{q}}{x^{2}}dx\right)  \notag \\
&\leq &\frac{\left\vert f^{\prime }\left( \frac{2ab}{a+b}\right) \right\vert
^{q}+\left\vert f^{\prime }\left( b\right) \right\vert ^{q}}{4}  \label{2-3b}
\end{eqnarray}%
and%
\begin{eqnarray}
\dint\limits_{1/2}^{1}\left\vert f^{\prime }\left( \frac{ab}{A_{t}}\right)
\right\vert ^{q}dt &\leq &\frac{1}{2}\left( \frac{2ab}{b-a}\dint\limits_{a}^{%
\frac{2ab}{a+b}}\frac{\left\vert f^{\prime }\left( x\right) \right\vert ^{q}%
}{x^{2}}dx\right)  \notag \\
&\leq &\frac{\left\vert f^{\prime }\left( \frac{2ab}{a+b}\right) \right\vert
^{q}+\left\vert f^{\prime }\left( a\right) \right\vert ^{q}}{4}.
\label{2-3c}
\end{eqnarray}%
A combination of (\ref{2-3a})-(\ref{2-3c}) gives the required inequality (%
\ref{2-3}).
\end{proof}

\begin{corollary}
Under the assumptions Theorem \ref{2.3} with $\lambda =0$, we have%
\begin{equation*}
\left\vert f\left( \frac{2ab}{a+b}\right) -\frac{ab}{b-a}\dint\limits_{a}^{b}%
\frac{f(x)}{x^{2}}dx\right\vert \leq \frac{ab\left( b-a\right) }{2}
\end{equation*}%
\begin{equation*}
\times \left\{ C_{1}^{\frac{1}{p}}(0,p;a,b)\left( \frac{\left\vert f^{\prime
}\left( \frac{2ab}{a+b}\right) \right\vert ^{q}+\left\vert f^{\prime }\left(
b\right) \right\vert ^{q}}{4}\right) ^{\frac{1}{q}}+C_{1}^{\frac{1}{p}%
}(0,p;b,a)\left( \frac{\left\vert f^{\prime }\left( \frac{2ab}{a+b}\right)
\right\vert ^{q}+\left\vert f^{\prime }\left( a\right) \right\vert ^{q}}{4}%
\right) ^{\frac{1}{q}}\right\} .
\end{equation*}
\end{corollary}

\begin{corollary}
Under the assumptions Theorem \ref{2.3} with $\lambda =1$, we have%
\begin{equation*}
\left\vert \frac{f(a)+f(b)}{2}-\frac{ab}{b-a}\dint\limits_{a}^{b}\frac{f(x)}{%
x^{2}}dx\right\vert \leq \frac{ab\left( b-a\right) }{2}
\end{equation*}%
\begin{equation*}
\times \left\{ C_{1}^{\frac{1}{p}}(1,p;a,b)\left( \frac{\left\vert f^{\prime
}\left( \frac{2ab}{a+b}\right) \right\vert ^{q}+\left\vert f^{\prime }\left(
b\right) \right\vert ^{q}}{4}\right) ^{\frac{1}{q}}+C_{1}^{\frac{1}{p}%
}(1,p;b,a)\left( \frac{\left\vert f^{\prime }\left( \frac{2ab}{a+b}\right)
\right\vert ^{q}+\left\vert f^{\prime }\left( a\right) \right\vert ^{q}}{4}%
\right) ^{\frac{1}{q}}\right\} .
\end{equation*}
\end{corollary}

\begin{corollary}
Under the assumptions Theorem \ref{2.3} with $\lambda =1/3$, we have%
\begin{equation*}
\left\vert \frac{1}{3}\left[ \frac{f(a)+f(b)}{2}+2f\left( \frac{2ab}{a+b}%
\right) \right] -\frac{ab}{b-a}\dint\limits_{a}^{b}\frac{f(x)}{x^{2}}%
dx\right\vert \leq \frac{ab\left( b-a\right) }{2}
\end{equation*}%
\begin{equation*}
\times \left\{ C_{1}^{\frac{1}{p}}(\frac{1}{3},p;a,b)\left( \frac{\left\vert
f^{\prime }\left( \frac{2ab}{a+b}\right) \right\vert ^{q}+\left\vert
f^{\prime }\left( b\right) \right\vert ^{q}}{4}\right) ^{\frac{1}{q}}+C_{1}^{%
\frac{1}{p}}(\frac{1}{3},p;b,a)\left( \frac{\left\vert f^{\prime }\left( 
\frac{2ab}{a+b}\right) \right\vert ^{q}+\left\vert f^{\prime }\left(
a\right) \right\vert ^{q}}{4}\right) ^{\frac{1}{q}}\right\}
\end{equation*}%
\bigskip
\end{corollary}

\begin{theorem}
\label{2.4}Let $f:I\subset \left( 0,\infty \right) \rightarrow 
\mathbb{R}
$ be a differentiable function on $I^{\circ }$, $a,b\in I$ with $a<b,$ and $%
f^{\prime }\in L[a,b].$ If $\left\vert f^{\prime }\right\vert ^{q}$ is
harmonically convex on $[a,b]$ for $q>1$ and then we have the following
inequality for $\lambda \in \left[ 0,1\right] $%
\begin{equation}
\left\vert \left( 1-\lambda \right) f\left( \frac{2ab}{a+b}\right) +\lambda
\left( \frac{f(a)+f(b)}{2}\right) -\frac{ab}{b-a}\dint\limits_{a}^{b}\frac{%
f(x)}{x^{2}}dx\right\vert \leq \frac{ab\left( b-a\right) }{4}  \label{2-4}
\end{equation}%
\begin{eqnarray*}
&&\times \frac{C_{4}^{1/p}(\lambda ,p)}{\left[ \left( 1-q\right) \left(
1-2q\right) \left( b-a\right) ^{2}\right] ^{1/q}}\left\{ \left(
C_{5}(q;a,b)\left\vert f^{\prime }\left( a\right) \right\vert
^{q}+C_{6}(q;a,b)\left\vert f^{\prime }\left( b\right) \right\vert
^{q}\right) ^{\frac{1}{q}}\right. \\
&&\left. +\left( C_{6}(q;b,a)\left\vert f^{\prime }\left( a\right)
\right\vert ^{q}+C_{5}(q;b,a)\left\vert f^{\prime }\left( b\right)
\right\vert ^{q}\right) ^{\frac{1}{q}}\right\} ,
\end{eqnarray*}%
where%
\begin{equation*}
C_{4}(\lambda ,p)=\frac{\lambda ^{p+1}+\left( 1-\lambda \right) ^{p+1}}{p+1},
\end{equation*}%
\begin{eqnarray*}
C_{5}(q;u,\vartheta ) &=&\left[ \left( \frac{u+\vartheta }{2}\right) ^{1-2q}%
\left[ \frac{\vartheta -3u}{2}-q(\vartheta -u)\right] +u^{2-2q}\right] , \\
C_{6}(q;u,\vartheta ) &=&\left[ \left( \frac{u+\vartheta }{2}\right) ^{1-2q}%
\left[ \frac{3\vartheta -u}{2}-q(\vartheta -u)\right] +u^{1-2q}\left[
u-2\vartheta +2q(\vartheta -u)\right] \right] ,\ u,\vartheta >0
\end{eqnarray*}%
and $1/p+1/q=1.$
\end{theorem}

\begin{proof}
Let $A_{t}=tb+(1-t)a$. Using Lemma \ref{2.1} and H\"{o}lder's integral
inequality, we deduce%
\begin{eqnarray*}
&&\left\vert \left( 1-\lambda \right) f\left( \frac{2ab}{a+b}\right)
+\lambda \left( \frac{f(a)+f(b)}{2}\right) -\frac{ab}{b-a}%
\dint\limits_{a}^{b}\frac{f(x)}{x^{2}}dx\right\vert \\
&\leq &\frac{ab\left( b-a\right) }{2}\left[ \dint\limits_{0}^{1/2}\frac{%
\left\vert \lambda -2t\right\vert }{A_{t}^{2}}\left\vert f^{\prime }\left( 
\frac{ab}{A_{t}}\right) \right\vert dt+\dint\limits_{1/2}^{1}\frac{%
\left\vert 2-\lambda -2t\right\vert }{A_{t}^{2}}\left\vert f^{\prime }\left( 
\frac{ab}{A_{t}}\right) \right\vert dt\right]
\end{eqnarray*}%
\begin{eqnarray}
&\leq &\frac{ab\left( b-a\right) }{2}\left\{ \left(
\dint\limits_{0}^{1/2}\left\vert \lambda -2t\right\vert ^{p}dt\right) ^{%
\frac{1}{p}}\left( \dint\limits_{0}^{1/2}\frac{1}{A_{t}^{2q}}\left\vert
f^{\prime }\left( \frac{ab}{A_{t}}\right) \right\vert ^{q}dt\right) ^{\frac{1%
}{q}}\right.  \label{2-4a} \\
&&\left. +\left( \dint\limits_{1/2}^{1}\left\vert 2-\lambda -2t\right\vert
^{p}dt\right) ^{\frac{1}{p}}\left( \dint\limits_{1/2}^{1}\frac{1}{A_{t}^{2q}}%
\left\vert f^{\prime }\left( \frac{ab}{A_{t}}\right) \right\vert
^{q}dt\right) ^{\frac{1}{q}}\right\} .  \notag
\end{eqnarray}%
Using the harmonically convexity of $\left\vert f^{\prime }\right\vert ^{q}$%
, we obtain%
\begin{equation*}
\dint\limits_{0}^{1/2}\frac{1}{A_{t}^{2q}}\left\vert f^{\prime }\left( \frac{%
ab}{A_{t}}\right) \right\vert ^{q}dt\leq \dint\limits_{0}^{1/2}\frac{%
t\left\vert f^{\prime }\left( a\right) \right\vert ^{q}+(1-t)\left\vert
f^{\prime }\left( b\right) \right\vert ^{q}}{A_{t}^{2q}}dt=\frac{1}{2\left(
1-q\right) \left( 1-2q\right) \left( b-a\right) ^{2}}
\end{equation*}%
\begin{eqnarray}
&&\times \left\{ \left[ \left( \frac{a+b}{2}\right) ^{1-2q}\left[ \frac{b-3a%
}{2}-q(b-a)\right] +a^{2-2q}\right] \left\vert f^{\prime }\left( a\right)
\right\vert ^{q}\right.  \label{2-4b} \\
&&\left. +\left[ \left( \frac{a+b}{2}\right) ^{1-2q}\left[ \frac{3b-a}{2}%
-q(b-a)\right] +a^{1-2q}\left[ a-2b+2q(b-a)\right] \right] \left\vert
f^{\prime }\left( b\right) \right\vert ^{q}\right\}  \notag
\end{eqnarray}%
and%
\begin{equation*}
\dint\limits_{1/2}^{1}\frac{1}{A_{t}^{2q}}\left\vert f^{\prime }\left( \frac{%
ab}{A_{t}}\right) \right\vert ^{q}dt\leq \dint\limits_{1/2}^{1}\frac{%
t\left\vert f^{\prime }\left( a\right) \right\vert ^{q}+(1-t)\left\vert
f^{\prime }\left( b\right) \right\vert ^{q}}{A_{t}^{2q}}dt=\frac{1}{2\left(
1-q\right) \left( 1-2q\right) \left( b-a\right) ^{2}}
\end{equation*}%
\begin{eqnarray}
&&\times \left\{ \left[ b^{1-2q}\left[ b-2a-2q(b-a)\right] +\left( \frac{a+b%
}{2}\right) ^{1-2q}\left[ \frac{3a-b}{2}+q(b-a)\right] \right] \left\vert
f^{\prime }\left( a\right) \right\vert ^{q}\right.  \label{2-4c} \\
&&\left. +\left[ \left( \frac{a+b}{2}\right) ^{1-2q}\left[ \frac{a-3b}{2}%
+q(b-a)\right] +b^{2-2q}\right] \left\vert f^{\prime }\left( b\right)
\right\vert ^{q}\right\}  \notag
\end{eqnarray}%
Further, we have%
\begin{equation}
\dint\limits_{0}^{1/2}\left\vert \lambda -2t\right\vert
^{p}dt=\dint\limits_{1/2}^{1}\left\vert 2-\lambda -2t\right\vert ^{p}dt=%
\frac{\lambda ^{p+1}+\left( 1-\lambda \right) ^{p+1}}{2\left( p+1\right) }
\label{2-4d}
\end{equation}%
A combination of (\ref{2-4a})-(\ref{2-4d}) gives the required inequality (%
\ref{2-4}).
\end{proof}

\begin{corollary}
Under the assumptions Theorem \ref{2.4} with $\lambda =0$, we have%
\begin{equation*}
\left\vert f\left( \frac{2ab}{a+b}\right) -\frac{ab}{b-a}\dint\limits_{a}^{b}%
\frac{f(x)}{x^{2}}dx\right\vert \leq \frac{ab\left( b-a\right) }{4\left(
p+1\right) ^{1/p}}
\end{equation*}%
\begin{eqnarray*}
&&\times \frac{1}{\left[ \left( 1-q\right) \left( 1-2q\right) \left(
b-a\right) ^{2}\right] ^{1/q}}\left\{ \left( C_{5}(q;a,b)\left\vert
f^{\prime }\left( a\right) \right\vert ^{q}+C_{6}(q;a,b)\left\vert f^{\prime
}\left( b\right) \right\vert ^{q}\right) ^{\frac{1}{q}}\right. \\
&&\left. +\left( C_{6}(q;b,a)\left\vert f^{\prime }\left( a\right)
\right\vert ^{q}+C_{5}(q;b,a)\left\vert f^{\prime }\left( b\right)
\right\vert ^{q}\right) ^{\frac{1}{q}}\right\} .
\end{eqnarray*}
\end{corollary}

\begin{corollary}
Under the assumptions Theorem \ref{2.4} with $\lambda =1$, we have%
\begin{equation*}
\left\vert \frac{f(a)+f(b)}{2}-\frac{ab}{b-a}\dint\limits_{a}^{b}\frac{f(x)}{%
x^{2}}dx\right\vert \leq \frac{ab\left( b-a\right) }{4\left( p+1\right)
^{1/p}}
\end{equation*}%
\begin{eqnarray*}
&&\times \frac{1}{\left[ \left( 1-q\right) \left( 1-2q\right) \left(
b-a\right) ^{2}\right] ^{1/q}}\left\{ \left( C_{5}(q;a,b)\left\vert
f^{\prime }\left( a\right) \right\vert ^{q}+C_{6}(q;a,b)\left\vert f^{\prime
}\left( b\right) \right\vert ^{q}\right) ^{\frac{1}{q}}\right. \\
&&\left. +\left( C_{6}(q;b,a)\left\vert f^{\prime }\left( a\right)
\right\vert ^{q}+C_{5}(q;b,a)\left\vert f^{\prime }\left( b\right)
\right\vert ^{q}\right) ^{\frac{1}{q}}\right\} .
\end{eqnarray*}
\end{corollary}

\begin{corollary}
Under the assumptions Theorem \ref{2.4} with $\lambda =1/3$, we have%
\begin{equation*}
\left\vert \frac{1}{3}\left[ \frac{f(a)+f(b)}{2}+2f\left( \frac{2ab}{a+b}%
\right) \right] -\frac{ab}{b-a}\dint\limits_{a}^{b}\frac{f(x)}{x^{2}}%
dx\right\vert \leq \frac{ab\left( b-a\right) }{4\left( 3^{p+1}\left(
p+1\right) \right) ^{1/p}}
\end{equation*}%
\begin{eqnarray*}
&&\times \frac{1+2^{p+1}}{\left[ \left( 1-q\right) \left( 1-2q\right) \left(
b-a\right) ^{2}\right] ^{1/q}}\left\{ \left( C_{5}(q;a,b)\left\vert
f^{\prime }\left( a\right) \right\vert ^{q}+C_{6}(q;a,b)\left\vert f^{\prime
}\left( b\right) \right\vert ^{q}\right) ^{\frac{1}{q}}\right. \\
&&\left. +\left( C_{6}(q;b,a)\left\vert f^{\prime }\left( a\right)
\right\vert ^{q}+C_{5}(q;b,a)\left\vert f^{\prime }\left( b\right)
\right\vert ^{q}\right) ^{\frac{1}{q}}\right\} .
\end{eqnarray*}
\end{corollary}

\section{Some applications for special means}

Let us recall the following special means of two nonnegative number $a,b$
with $b>a:$

\begin{enumerate}
\item The arithmetic mean%
\begin{equation*}
A=A\left( a,b\right) :=\frac{a+b}{2}.
\end{equation*}

\item The geometric mean%
\begin{equation*}
G=G\left( a,b\right) :=\sqrt{ab}.
\end{equation*}

\item The harmonic mean%
\begin{equation*}
H=H\left( a,b\right) :=\frac{2ab}{a+b}.
\end{equation*}

\item The Logarithmic mean%
\begin{equation*}
L=L\left( a,b\right) :=\frac{b-a}{\ln b-\ln a}.
\end{equation*}

\item The p-Logarithmic mean%
\begin{equation*}
L_{p}=L_{p}\left( a,b\right) :=\left( \frac{b^{p+1}-a^{p+1}}{(p+1)(b-a)}%
\right) ^{\frac{1}{p}},\ \ p\in 
\mathbb{R}
\backslash \left\{ -1,0\right\} .
\end{equation*}

\item the Identric mean%
\begin{equation*}
I=I\left( a,b\right) =\frac{1}{e}\left( \frac{b^{b}}{a^{a}}\right) ^{\frac{1%
}{b-a}}.
\end{equation*}
\end{enumerate}

These means are often used in numerical approximation and in other areas.
However, the following simple relationships are known in the literature:%
\begin{equation*}
H\leq G\leq L\leq I\leq A.
\end{equation*}%
It is also known that $L_{p}$ is monotonically increasing over $p\in 
\mathbb{R}
,$ denoting $L_{0}=I$ and $L_{-1}=L.$

\begin{proposition}
Let $0<a<b$ and $\lambda \in \left[ 0,1\right] $. Then we have the following
inequality%
\begin{equation*}
\left\vert \left( 1-\lambda \right) H+\lambda A-\frac{G^{2}}{L}\right\vert
\leq \frac{ab\left( b-a\right) }{2}\left\{ C_{1}(\lambda ;a,b)+C_{1}(\lambda
;b,a)\right\} ,
\end{equation*}%
where $C_{1}$ is defined as in Theorem \ref{2.2}.
\end{proposition}

\begin{proof}
The assertion follows from the inequality (\ref{2-2}) in Theorem \ref{2.2},
for $f:\left( 0,\infty \right) \rightarrow 
\mathbb{R}
,\ f(x)=x.$
\end{proof}

\begin{proposition}
Let $0<a<b$ and $\lambda \in \left[ 0,1\right] $. Then we have the following
inequality%
\begin{equation*}
\left\vert \left( 1-\lambda \right) H+\lambda A-\frac{G^{2}}{L}\right\vert
\leq \frac{ab\left( b-a\right) }{2^{1+1/q}}\left\{ C_{1}^{\frac{1}{p}%
}(\lambda ,p;a,b)+C_{1}^{\frac{1}{p}}(\lambda ,p;b,a)\right\} ,
\end{equation*}%
where $q>1$, $1/p+1/q=1$ and $C_{1}$ is defined as in Theorem \ref{2.3}.
\end{proposition}

\begin{proof}
The assertion follows from the inequality (\ref{2-3}) in Theorem \ref{2.3},
for $f:\left( 0,\infty \right) \rightarrow 
\mathbb{R}
,\ f(x)=x.$
\end{proof}

\begin{proposition}
Let $0<a<b$ and $\lambda \in \left[ 0,1\right] $. Then we have the following
inequality%
\begin{equation*}
\left\vert \left( 1-\lambda \right) H+\lambda A-\frac{G^{2}}{L}\right\vert
\leq \frac{ab\left( b-a\right) C_{4}^{1/p}(\lambda ,p)}{4\left[ \left(
1-q\right) \left( 1-2q\right) \left( b-a\right) ^{2}\right] ^{1/q}}
\end{equation*}%
\begin{equation*}
\times \left\{ \left( C_{5}(q;a,b)+C_{6}(q;a,b)\right) ^{\frac{1}{q}}+\left(
C_{6}(q;b,a)+C_{5}(q;b,a)\right) ^{\frac{1}{q}}\right\} ,
\end{equation*}%
where $q>1$, $1/p+1/q=1$ and $C_{4,}C_{5}$ and $C_{6}$ are defined as in
Theorem \ref{2.4}.
\end{proposition}

\begin{proof}
The assertion follows from the inequality (\ref{2-4}) in Theorem \ref{2.4},
for $f:\left( 0,\infty \right) \rightarrow 
\mathbb{R}
,\ f(x)=x.$
\end{proof}

\begin{proposition}
Let $0<a<b$, $\lambda \in \left[ 0,1\right] $ and $q\geq 1.$ Then we have
the following inequality%
\begin{eqnarray*}
&&\left\vert \left( 1-\lambda \right) H^{2}+\lambda
A(a^{2},b^{2})-G^{2}\right\vert  \\
&\leq &ab\left( b-a\right) \left\{ C_{1}^{1-\frac{1}{q}}(\lambda ;a,b)\left[
C_{2}(\lambda ;a,b)a^{q}+C_{3}(\lambda ;a,b)b^{q}\right] ^{\frac{1}{q}%
}\right.  \\
&&\left. +C_{1}^{1-\frac{1}{q}}(\lambda ;b,a)\left[ C_{3}(\lambda
;b,a)a^{q}+C_{2}(\lambda ;b,a)b^{q}\right] ^{\frac{1}{q}}\right\} ,
\end{eqnarray*}%
where $C_{1},C_{2}$ and $C_{3}$ are defined as in Theorem \ref{2.2}.
\end{proposition}

\begin{proof}
The assertion follows from the inequality (\ref{2-2}) in Theorem \ref{2.2},
for $f:\left( 0,\infty \right) \rightarrow 
\mathbb{R}
,\ f(x)=x^{2}.$
\end{proof}

\begin{proposition}
Let $0<a<b$ and $\lambda \in \left[ 0,1\right] $. Then we have the following
inequality%
\begin{eqnarray*}
&&\left\vert \left( 1-\lambda \right) H^{2}+\lambda
A(a^{2},b^{2})-G^{2}\right\vert \leq \frac{ab\left( b-a\right) }{2^{1/q}} \\
&&\times \left\{ C_{1}^{\frac{1}{p}}(\lambda ,p;a,b)A^{\frac{1}{q}}\left(
H^{q},b^{q}\right) +C_{1}^{\frac{1}{p}}(\lambda ,p;b,a)A^{\frac{1}{q}}\left(
a^{q},H^{q}\right) \right\} ,
\end{eqnarray*}%
where $q>1$ and $1/p+1/q=1.$
\end{proposition}

\begin{proof}
The assertion follows from the inequality (\ref{2-3}) in Theorem \ref{2.3},
for $f:\left( 0,\infty \right) \rightarrow 
\mathbb{R}
,\ f(x)=x^{2}.$
\end{proof}

\begin{proposition}
Let $0<a<b$ and $\lambda \in \left[ 0,1\right] $. Then we have the following
inequality%
\begin{eqnarray*}
&&\left\vert \left( 1-\lambda \right) H^{2}+\lambda
A(a^{2},b^{2})-G^{2}\right\vert \leq \frac{ab\left( b-a\right)
C_{4}^{1/p}(\lambda ,p)}{2\left[ \left( 1-q\right) \left( 1-2q\right) \left(
b-a\right) ^{2}\right] ^{1/q}} \\
&&\times \left\{ \left( C_{5}(q;a,b)a^{q}+C_{6}(q;a,b)b^{q}\right) ^{\frac{1%
}{q}}+\left( C_{6}(q;b,a)a^{q}+C_{5}(q;b,a)b^{q}\right) ^{\frac{1}{q}%
}\right\} ,
\end{eqnarray*}%
where $q>1$, $1/p+1/q=1$ and $C_{4,}C_{5}$ and $C_{6}$ are defined as in
Theorem \ref{2.4}.
\end{proposition}

\begin{proof}
The assertion follows from the inequality (\ref{2-4}) in Theorem \ref{2.4},
for $f:\left( 0,\infty \right) \rightarrow 
\mathbb{R}
,\ f(x)=x^{2}.$
\end{proof}

\begin{proposition}
Let $0<a<b$, $n\in \left( -1,\infty \right) \backslash \left\{ 0\right\} $, $%
\lambda \in \left[ 0,1\right] $ and $q\geq 1.$ Then we have the following
inequality%
\begin{equation*}
\left\vert \left( 1-\lambda \right) H^{n+2}+\lambda
A(a^{n+2},b^{n+2})-G^{2}.L_{n}^{n}\right\vert 
\end{equation*}%
\begin{eqnarray*}
&\leq &\frac{ab\left( b-a\right) (n+2)}{2}\left\{ C_{1}^{1-\frac{1}{q}%
}(\lambda ;a,b)\left[ C_{2}(\lambda ;a,b)a^{\left( n+1\right)
q}+C_{3}(\lambda ;a,b)b^{\left( n+1\right) q}\right] ^{\frac{1}{q}}\right. 
\\
&&\left. +C_{1}^{1-\frac{1}{q}}(\lambda ;b,a)\left[ C_{3}(\lambda
;b,a)a^{\left( n+1\right) q}+C_{2}(\lambda ;b,a)b^{\left( n+1\right) q}%
\right] ^{\frac{1}{q}}\right\} ,
\end{eqnarray*}%
where $C_{1},C_{2}$ and $C_{3}$ are defined as in Theorem \ref{2.2}.
\end{proposition}

\begin{proof}
The assertion follows from the inequality (\ref{2-2}) in Theorem \ref{2.2},
for $f:\left( 0,\infty \right) \rightarrow 
\mathbb{R}
,\ f(x)=x^{n+2},\ n\in \left( -1,\infty \right) \backslash \left\{ 0\right\}
.$
\end{proof}

\begin{proposition}
Let $0<a<b$ and $n\in \left( -1,\infty \right) \backslash \left\{ 0\right\} .
$ Then we have the following inequality%
\begin{equation*}
\left\vert \left( 1-\lambda \right) H^{n+2}+\lambda
A(a^{n+2},b^{n+2})-G^{2}.L_{n}^{n}\right\vert \leq \frac{ab\left( b-a\right)
(n+2)}{2^{1+1/q}}
\end{equation*}%
\begin{equation*}
\times \left\{ C_{1}^{1-\frac{1}{q}}(\lambda ;a,b)A^{\frac{1}{q}}\left(
H^{(n+1)q},b^{(n+1)q}\right) +C_{1}^{1-\frac{1}{q}}(\lambda ;b,a)A^{\frac{1}{%
q}}\left( a^{(n+1)q},H^{(n+1)q}\right) \right\} ,
\end{equation*}%
where $q>1$, $1/p+1/q=1$ and $C_{1}$ is defined as in Theorem \ref{2.3}.
\end{proposition}

\begin{proof}
The assertion follows from the inequality (\ref{2-3}) in Theorem \ref{2.3},
for $f:\left( 0,\infty \right) \rightarrow 
\mathbb{R}
,\ f(x)=x^{n+2},\ n\in \left( -1,\infty \right) \backslash \left\{ 0\right\}
.$
\end{proof}

\begin{proposition}
Let $0<a<b$, $\lambda \in \left[ 0,1\right] $. and $n\in \left( -1,\infty
\right) \backslash \left\{ 0\right\} .$ Then we have the following inequality%
\begin{equation*}
\left\vert \left( 1-\lambda \right) H^{n+2}+\lambda
A(a^{n+2},b^{n+2})-G^{2}.L_{n}^{n}\right\vert 
\end{equation*}%
\begin{eqnarray*}
&\leq &\frac{ab\left( b-a\right) (n+2)C_{4}^{1/p}(\lambda ,p)}{4\left[
\left( 1-q\right) \left( 1-2q\right) \left( b-a\right) ^{2}\right] ^{1/q}}%
\left\{ \left( C_{5}(q;a,b)a^{(n+1)q}+C_{6}(q;a,b)b^{(n+1)q}\right) ^{\frac{1%
}{q}}\right.  \\
&&\left. +\left( C_{6}(q;b,a)a^{(n+1)q}+C_{5}(q;b,a)b^{(n+1)q}\right) ^{%
\frac{1}{q}}\right\} ,
\end{eqnarray*}%
where $q>1$, $1/p+1/q=1$ and $C_{5}$ and $C_{6}$ are defined as in Theorem %
\ref{2.4}.
\end{proposition}

\begin{proof}
The assertion follows from the inequality (\ref{2-4}) in Theorem \ref{2.4},
for $f:\left( 0,\infty \right) \rightarrow 
\mathbb{R}
,\ f(x)=x^{n+2},\ n\left( -1,\infty \right) \backslash \left\{ 0\right\} .$
\end{proof}

\begin{proposition}
Let $0<a<b$, $\lambda \in \left[ 0,1\right] $ and $q\geq 1.$ Then we have
the following inequality%
\begin{equation*}
\left\vert \left( 1-\lambda \right) H^{2}\ln H+\lambda A\left( a^{2}\ln
a,b^{2}\ln b\right) -G^{2}\ln I\right\vert \leq ab\left( b-a\right) 
\end{equation*}%
\begin{eqnarray*}
&&\times \left\{ C_{1}^{1-\frac{1}{q}}(\lambda ;a,b)\left[ C_{2}(\lambda
;a,b)G^{2q}\left( a,A(1,\ln a)\right) +C_{3}(\lambda ;a,b)G^{2q}\left(
b,A(1,\ln b)\right) \right] ^{\frac{1}{q}}\right.  \\
&&\left. +C_{1}^{1-\frac{1}{q}}(\lambda ;b,a)\left[ C_{3}(\lambda
;b,a)G^{2q}\left( a,A(1,\ln a)\right) +C_{2}(\lambda ;b,a)G^{2q}\left(
b,A(1,\ln b)\right) \right] ^{\frac{1}{q}}\right\} ,
\end{eqnarray*}%
where $C_{1},C_{2}$ and $C_{3}$ are defined as in Theorem \ref{2.2}.
\end{proposition}

\begin{proof}
The assertion follows from the inequality (\ref{2-2}) in Theorem \ref{2.2},
for $f:\left( 0,\infty \right) \rightarrow 
\mathbb{R}
,\ f(x)=x^{2}\ln x.$
\end{proof}

\begin{proposition}
Let $0<a<b$, and $\lambda \in \left[ 0,1\right] $. Then we have the
following inequality%
\begin{equation*}
\left\vert \left( 1-\lambda \right) H^{2}\ln H+\lambda A\left( a^{2}\ln
a,b^{2}\ln b\right) -G^{2}\ln I\right\vert \leq \frac{ab\left( b-a\right) }{%
2^{1/q}}
\end{equation*}%
\begin{eqnarray*}
&&\times \left\{ C_{1}^{\frac{1}{p}}(0,p;a,b)A^{\frac{1}{q}}\left(
G^{2q}\left( H,A(1,\ln H)\right) ,G^{2q}\left( b,A(1,\ln b)\right) \right)
\right.  \\
&&\left. +C_{1}^{\frac{1}{p}}(0,p;b,a)A^{\frac{1}{q}}\left( G^{2q}\left(
H,A(1,\ln H)\right) ,G^{2q}\left( a,A(1,\ln a)\right) \right) \right\} ,
\end{eqnarray*}%
where $q>1$, $1/p+1/q=1$ and $C_{1}$ is defined as in Theorem \ref{2.3}.
\end{proposition}

\begin{proof}
The assertion follows from the inequality (\ref{2-3}) in Theorem \ref{2.3},
for $f:\left( 0,\infty \right) \rightarrow 
\mathbb{R}
,\ f(x)=x^{2}\ln x.$
\end{proof}

\begin{proposition}
Let $0<a<b$, and $\lambda \in \left[ 0,1\right] $. Then we have the
following inequality%
\begin{equation*}
\left\vert \left( 1-\lambda \right) H^{2}\ln H+\lambda A\left( a^{2}\ln
a,b^{2}\ln b\right) -G^{2}\ln I\right\vert \leq \frac{ab\left( b-a\right) }{2%
}
\end{equation*}%
\begin{eqnarray*}
&&\frac{C_{4}^{1/p}(\lambda ,p)}{\left[ \left( 1-q\right) \left( 1-2q\right)
\left( b-a\right) ^{2}\right] ^{1/q}}\left\{ \left( C_{5}(q;a,b)G^{2q}\left(
a,A(1,\ln a)\right) +C_{6}(q;a,b)G^{2q}\left( b,A(1,\ln b)\right) \right) ^{%
\frac{1}{q}}\right.  \\
&&\left. +\left( C_{6}(q;b,a)G^{2q}\left( a,A(1,\ln a)\right)
+C_{5}(q;b,a)G^{2q}\left( b,A(1,\ln b)\right) \right) ^{\frac{1}{q}}\right\}
,
\end{eqnarray*}%
where $q>1$, $1/p+1/q=1$ and $C_{5}$ and $C_{6}$ are defined as in Theorem %
\ref{2.4}.
\end{proposition}

\begin{proof}
The assertion follows from the inequality (\ref{2-4}) in Theorem \ref{2.4},
for $f:\left( 0,\infty \right) \rightarrow 
\mathbb{R}
,\ f(x)=x^{2}\ln x.$
\end{proof}

\end{document}